\documentclass[a4paper,12pt]{amsart}

\usepackage{amssymb}
\usepackage{latexsym}
\usepackage{amsmath}
\usepackage{amscd}
\usepackage{graphicx}
\usepackage{palatino}

\title[secondary characteristic classes]
{Secondary characteristic classes for subgroups of
automorphism groups of free groups}

\author{Shigeyuki Morita}
\address{Graduate School of Mathematical Sciences, 
The University of Tokyo, 
3-8-1 Komaba, 
Meguro-ku, Tokyo, 153-8914, Japan}
\email{morita@ms.u-tokyo.ac.jp}
\author{Takuya Sakasai}
\address{Graduate School of Mathematical Sciences, 
The University of Tokyo, 
3-8-1 Komaba, 
Meguro-ku, Tokyo, 153-8914, Japan}
\email{sakasai@ms.u-tokyo.ac.jp}
\author{Masaaki Suzuki}
\address{Department of Frontier Media Science, 
Meiji University, 
4-21-1 Nakano, Nakano-ku, Tokyo, 164-8525, Japan}
\email{macky@fms.meiji.ac.jp}

\thanks{
The authors were partially supported by KAKENHI (No.15H03618), 
Japan Society for the Promotion of Science, 
Japan.
}


\newtheorem{thm}{Theorem}[section]
\newtheorem{prop}[thm]{Proposition}

\theoremstyle{definition}
\newtheorem{definition}[thm]{Definition}

\newtheorem{remark}[thm]{Remark}
\newtheorem{problem}[thm]{Problem}
\newtheorem{conj}[thm]{Conjecture}

\newtheorem{quest}[thm]{Question}

\begin{document}

\newcommand{\Mg}{\mathcal{M}_g}
\newcommand{\Mgp}{\mathcal{M}_{g,\ast}}
\newcommand{\Mgb}{\mathcal{M}_{g,1}}

\newcommand{\hg}{\mathfrak{h}_{g,1}}
\newcommand{\ag}{\mathfrak{a}_g}
\newcommand{\Ln}{\mathcal{L}_n}

\newcommand{\Sg}{\Sigma_g}
\newcommand{\Sgb}{\Sigma_{g,1}}
\newcommand{\la}{\lambda}

\newcommand{\Symp}[1]{Sp(2g,\mathbb{#1})}
\newcommand{\symp}[1]{\mathfrak{sp}(2g,\mathbb{#1})}
\newcommand{\gl}[1]{\mathfrak{gl}(n,\mathbb{#1})}

\newcommand{\At}[1]{\mathcal{A}_{#1}^t (H)}
\newcommand{\Hq}{H_{\mathbb{Q}}}

\newcommand{\Ker}{\mathop{\mathrm{Ker}}\nolimits}
\newcommand{\Hom}{\mathop{\mathrm{Hom}}\nolimits}
\renewcommand{\Im}{\mathop{\mathrm{Im}}\nolimits}

\newcommand{\Der}{\mathop{\mathrm{Der}}\nolimits}
\newcommand{\Out}{\mathop{\mathrm{Out}}\nolimits}
\newcommand{\Aut}{\mathop{\mathrm{Aut}}\nolimits}
\newcommand{\Q}{\mathbb{Q}}
\newcommand{\Z}{\mathbb{Z}}
\newcommand{\R}{\mathbb{R}}
\newcommand{\C}{\mathbb{C}}

\begin{abstract}

By analyzing how the Borel regulator classes vanish on
various groups related to $\mathrm{GL}(n,\Z)$, we define 
{\it three} series of 
secondary characteristic classes for subgroups of 
automorphism groups of free groups.

The first case is the $\mathrm{IA}$-automorphism groups and we show
that our 
classes coincide with higher $\mathrm{FR}$
torsions due to Igusa. The second case is the mapping class groups
and our classes also turn out to be his higher torsions 
which are non-zero multiples of
the Mumford-Morita-Miller classes of {\it even} indices. Our construction gives
new group cocycles for 
these still mysterious classes.
The third case is the outer automorphism groups of free groups
of specific ranks. Here we give a conjectural geometric meaning 
to a series of unstable homology classes called the Morita classes.
We expect that certain {\it unstable} secondary classes would detect them.

\end{abstract}

\renewcommand\baselinestretch{1.1}
\setlength{\baselineskip}{16pt}

\newcounter{fig}
\setcounter{fig}{0}

\maketitle

\section{Introduction and statements of the main results}\label{sec:intro}

The Borel regulator classes $\beta_{2k+1}\in H^{4k+1}(\mathrm{GL}(n,\Z);\R)\ (k=1,2,\ldots)$
are stable cohomology classes of $\mathrm{GL}(n,\Z)$ and they play fundamental 
roles in diverse branches of mathematics including number theory, algebraic geometry,
differential geometry and topology or more specifically algebraic $K$-theory and characteristic classes of
flat bundles with arithmetic structure groups.

There are two important groups related to $\mathrm{GL}(n,\Z)$. One is the outer automorphism
group of a free group of rank $n$, denoted by $\mathrm{Out}\,F_n$ and there exists a
natural homomorphism from this group onto $\mathrm{GL}(n,\Z)$. The other is the
integral symplectic group $\mathrm{Sp}(2g,\Z)$ which can be considered as a natural
subgroup of $\mathrm{GL}(2g,\Z)$.

Igusa \cite{i} proved that the pull back of all the Borel classes to $\mathrm{Out}\,F_n$ vanish
and, more recently, Galatius \cite{ga} proved a definitive 
result that the stable rational cohomology of $\mathrm{Out}\,F_n$ is trivial.
On the other hand, a classical theorem of Borel \cite{borel1} 
which determines the stable rational cohomology groups of 
both $\mathrm{GL}(n,\Z)$ and $\mathrm{Sp}(2g,\Z)$ 
implies that the restriction of the Borel classes to the latter group vanish.

By making use of these vanishing of the Borel regulator classes $\beta_{2k+1}$
in various ways, we define 
three series of 
secondary characteristic classes for subgroups of 
the automorphism groups of free groups.

First, in $\S 3$ we use the vanishing of $\beta_{2k+1}$ on $\mathrm{Out}\,F_n$,
due to Igusa and Galatius, to define a secondary class 
$$
T\beta_{2k+1}\in H^{4k}(\mathrm{IOut}_n;\R)
$$ 
(see Definition \ref{def:Tb}) , where $\mathrm{IOut}_n$ denotes
the subgroup of $\mathrm{Out}\, F_n$ consisting of all the elements which act
on the abelianization of $F_n$ trivially.

Second, in $\S 4$ we compare the above vanishing on $\mathrm{Out}\,F_n$ with that on
$\mathrm{Sp}(2g,\Z)$ for the case $n=2g$ to define another secondary class
$$
\hat{\beta}_{2k+1}\in H^{4k}(\mathcal{M}_{g,*};\R)
$$ 
(see Definition \ref{def:hb}), where $\mathcal{M}_{g,*}$ denotes the mapping class group of a genus $g$ closed
surface with a base point and it
can be considered as a subgroup of $\mathrm{Out}\, F_{2g}$.

We show that $T\beta_{2k+1}$ is nothing other than the higher $\mathrm{FR}$
torsion $\tau_{2k}(\mathrm{IOut}_n)$ of the group $\mathrm{IOut}_n$
due to Igusa \cite{i} (see Theorem \ref{thm:Tb}). We also show that $\hat{\beta}_{2k+1}$ is 
the same as his higher torsion class and this implies that it is a
non-zero 
multiple of the
$\mathrm{MMM}$ class $e_{2k}$
of even indices 
(see Theorem \ref{thm:hb}).
Our construction gives a direct
relation between the Borel classes and the $\mathrm{MMM}$ classes of even indices.

Finally in $\S 5$, which is the ``heart" of the present paper,  we give a conjectural geometric meaning of the Morita classes
$\mu_k$, which make a series of unstable homology classes of
$\mathrm{Out}\,F_n$. This is done by considering  certain {\it unstable} refinements of
the above construction (see Conjecture \ref{conj:mu} and Conjecture \ref{conj:mud}). 
This would give a possible way of proving
non-triviality of these classes. At present, only the first 
three classes are known to be non-trivial (\cite{morita99}\cite{cov}\cite{gr}).
We hope that the results in $\S 3$ and $\S 4$ would serve as supporting evidence for the 
above conjectures.

The present work began by a trial
to prove our conjecture on possible geometric meaning of the Morita classes
as sketched in Remark 9.6 of \cite{mss2} and further discussed in $\S 5$ of the present paper.
Meanwhile we found that, by using simpler ideas we can define certain secondary characteristic classes
for the $\mathrm{IA}$ automorphism groups as well as the mapping class groups. Then we noticed that
our secondary classes are nothing other than the higher $\mathrm{FR}$ torsion classes
due to Igusa (and Klein for the Torelli group)
developed in \cite{i} as stated above. Our results could have been obtained
right after Galatius' paper \cite{ga} appeared.
Indeed, it may be said that our construction ``realizes" Igusa's
higher $\mathrm{FR}$ torsions from the viewpoint of the theory of group cohomology
in the cases of the above two kinds of groups, based on the vanishing theorem of Galatius.

We hope that our construction would shed a new light on the difficult open 
problems of determining whether Igusa's higher torsion classes as well
as the $\mathrm{MMM}$ classes of even indices are non-trivial on
$\mathrm{IOut}_n$ and the Torelli group, respectively.
These two classes are {\it stable} characteristic classes in the sense that
they are defined for all ranks or genera.
On the other hand, the initial problem of giving a geometric meaning to the Morita 
classes treats the exact boundary line between the stable and unstable ranges
for the Borel classes.

\section{Preliminaries}\label{sec:prl}

In this section, we recall basic theorems concerning the stable cohomology of
various groups which we use in this paper. 
They are, $\mathrm{GL}(n,\Z),
\mathrm{Sp}(2g,\Z), \mathrm{Aut}\,F_n$ and the mapping class group
$\mathcal{M}_g$.
First, we have the following classical result.

\begin{thm}[Borel \cite{borel1}]
The cohomology groups $H^*(\mathrm{GL}(n,\Z);\Q)$ and $H^*(\mathrm{Sp}(2g,\Z);\Q)$
stabilize, with respect to $n$ and $g$ respectively, and the stable cohomology groups are
given as follows.
\begin{align*}
\lim_{n\to\infty} H^*(\mathrm{GL}(n,\Z);\R)&\cong \wedge_\R(\beta_3,\beta_5,\ldots)\\
\lim_{n\to\infty} H^*(\mathrm{Sp}(n,\Z);\Q)&\cong \Q[c_1,c_3,\ldots].
\end{align*}
\label{thm:borel}
\end{thm}

Here
$$
\beta_{2k+1}\in H^{4k+1}(\mathrm{GL}(n,\Z);\R)\ (k=1,2,\ldots)
$$
denote the Borel regulator classes and
$c_1, c_3,\ldots, c_{2i-1},\ldots\in H^*(\mathrm{Sp}(2g,\Z);\Q)$ denote the Chern classes of the universal 
$g$-dimensional complex vector bundle over the classifying space of the group $\mathrm{Sp}(2g,\Z)$.
The Borel classes are stable classes in the sense that they
are pull backs of the Borel regulator classes 
$\beta_k\in H^{2k-1}(\mathrm{BGL}(\infty,\C)^\delta;\C)\ (k=1,2,\ldots)$
under the natural inclusion
$
i: \mathrm{GL}(n,\Z)\rightarrow \mathrm{GL}(\infty,\C)^\delta
$
where 
$$
\mathrm{GL}(\infty,\C)^\delta=\lim_{n\to\infty} \mathrm{GL}(n,\C)^\delta
$$
equipped with the {\it discrete} topology. 
See \cite{d}\cite{cs} for geometric background of these classes. 
We recall a property of the Borel classes for later use.
\begin{prop}[see Dupont-Hain-Zucker \cite{dhz}]
The Borel regulator class $\beta_{2k+1}$ is primitive. Namely, 
if we denote by
$$
\bar{\mu}: \mathrm{GL}(n,\Z)\times \mathrm{GL}(n',\Z)\rightarrow \mathrm{GL}(n+n',\Z)
$$
the natural inclusion, then we have
$$
\bar{\mu}^*(\beta_{2k+1})=\beta_{2k+1}\otimes 1+1\otimes \beta_{2k+1}.
$$
\label{prop:betap}
\end{prop}

Next we consider the (outer) automorphism groups of free groups
$\mathrm{Aut}\, F_n$ and $\mathrm{Out}\,F_n$.
Hatcher and Vogtmann \cite{HV} (see also \cite{HVW})
proved that homology groups of these groups stabilize and 
Galatius determined the stable homology groups as follows.
Let $\mathrm{Aut}_\infty$ denote $\lim_{n\to\infty} \mathrm{Aut}\,F_n$
and let $QS^0=\lim_{n\to\infty} \Omega^n S^n$.
\begin{thm}[Galatius \cite{ga}]
There exists a homology equivalence (H.e. for short)
$$
\Z\times\mathrm{BAut}_\infty\overset{H.e.}{\cong}QS^0.
$$
In particular
$$
\lim_{n\to\infty} H^*(\mathrm{Aut}\, F_n;\Q)\cong \lim_{n\to\infty} H^*(\mathrm{Out}\, F_n;\Q)\cong\Q.
$$
\label{thm:galatius}
\end{thm}
Finally we consider the mapping class groups. Let $\mathcal{M}_g$ denote the mapping class
group of a closed oriented surface $\Sigma_g$ of genus $g$. Also let $\mathcal{M}_{g,*}$
and $\mathcal{M}_{g,1}$ be the mapping class groups of $\Sigma_g$ relative to
a base point and embedded disk, respectively.
Harer \cite{harer} proved that the (co)homology groups of the mapping class groups
stabilize and Madsen and Weiss determined the stable rational cohomology as follows.
\begin{thm}[Madsen-Weiss \cite{mw}]
$$
\lim_{g\to\infty} H^*(\mathcal{M}_g;\Q)\cong \lim_{g\to\infty} H^*(\mathcal{M}_{g,1};\Q)
\cong \Q[\text{$\mathrm{MMM}$ classes}].
$$
\label{thm:mw}
\end{thm}
We also have the following result.
\begin{thm}[see Harer \cite{harerthird}, Looijenga \cite{looijenga} and Madsen-Weiss \cite{mw}]
$$
\lim_{g\to\infty} H^*(\mathcal{M}_{g,*};\Q)
\cong \Q[e, \text{$\mathrm{MMM}$ classes}].
$$
\label{thm:looijenga}
\end{thm}
As for the stable ranges, we refer to \cite{borel1}\cite{vdKallen}\cite{suslin} for $\mathrm{GL}(n,\Z)$,
\cite{borel1}\cite{charney} for $\mathrm{Sp}(2g,\Z)$, \cite{HV}\cite{HVW} for the (outer) automorphism
groups of free groups
and \cite{wahl1} as well as references therein for the mapping class groups.

\section{Secondary characteristic classes for the $\mathrm{IA}$ automorphism group}\label{sec:scia}

In this section, we consider the $\mathrm{IA}$ automorphism group, denoted by $\mathrm{IA}_n$ 
which is defined to be the kernel of the natural projection $p: \mathrm{Aut}\, F_n\rightarrow \mathrm{GL}(n,\Z)$.
Thus we have the following exact sequence
$$
1\rightarrow \mathrm{IA}_n\overset{i}{\rightarrow}\mathrm{Aut}\, F_n\overset{p}{\rightarrow}
 \mathrm{GL}(n,\Z)\rightarrow 1.
$$
The outer automorphism group of $F_n$ is defined as $\mathrm{Out}\,F_n=\mathrm{Aut}\,F_n/\mathrm{Inn}\, F_n$
and we have the following  similar exact sequence
$$
1\rightarrow \mathrm{IOut}_n\overset{i}{\rightarrow}\mathrm{Out}\, F_n\overset{p}{\rightarrow}
 \mathrm{GL}(n,\Z)\rightarrow 1.
$$

The first of our three secondary characteristic classes are elements of
$H^*(\mathrm{IA}_n;\R)$ and $H^*(\mathrm{IOut}_n;\R)$ defined as follows.
Let
$$
b_{2k+1}\in Z^{4k+1}(\mathrm{GL}(n,\Z);\R)
$$
be a $(4k + 1)$-cocycle of the group $\mathrm{GL}(n,\Z)$ which represents the Borel 
regulator class
$\beta_{2k+1}\in H^{4k+1}(\mathrm{GL}(n,\Z);\R)$.
We know by Igusa that 
$$
p^*\beta_{2k+1}=0\in H^{4k+1}(\mathrm{Out}\,F_n;\R).
$$
Hence we can choose a $4k$-cochain
$$
z_{4k}\in C^{4k}(\mathrm{Out}\, F_n;\R)
$$
such that $\delta z_{4k}= p^*b_{2k+1}$. The restriction $i^* z_{4k}$ of $z_{4k}$
to the subgroup $\mathrm{IOut}_n\overset{i}{\subset}\mathrm{Out}\,F_n$ is a
cocycle and we can consider its cohomology class
$$
[i^* z_{4k}]\in H^{4k}(\mathrm{IOut}_n;\R).
$$

\begin{prop}
The cohomology class $[i^* z_{4k}]\in H^{4k}(\mathrm{IOut}_n;\R)$ 
is well-defined independent of the choices of $b_{2k+1}$ and $z_{4k}$
in the stable range $n\geq 2k+4$ of $\mathrm{Out}\,F_n$.
Furthermore it is $\mathrm{GL}(n,\Z)$-invariant so that
$$
[i^* z_{4k}]\in H^{4k}(\mathrm{IOut}_n;\R)^{\mathrm{GL}(n,\Z)}.
$$
\label{prop:binv}
\end{prop}

\begin{proof}
First we prove the former part of the claim.
Let $b'_{2k+1}\in Z^{4k+1}(\mathrm{GL}(n,\Z);\R)$ be another representative of 
the Borel class $\beta_{2k+1}$ and let
$$
z'_{4k}\in C^{4k}(\mathrm{Out}\, F_n;\R)
$$
be a cochain such that $\delta z'_{4k}= p^*b'_{2k+1}$.  Now there exists an
element $u\in C^{4k}(\mathrm{GL}(n,\Z);\R)$ such that
$$
\delta u=b'_{2k+1}-b_{2k+1}.
$$
Then we have
$$
\delta z'_{4k}=p^*b'_{2k+1}=p^*(b_{2k+1}+\delta u)=\delta(z_{4k}+p^*u).
$$
It follows that
$$
\delta(z'_{4k}-z_{4k}-p^*u)=0.
$$
By the vanishing theorem of Galatius, there exists an element $v\in C^{4k-1}(\mathrm{Out}\,F_n;\R)$ such that
$$
z'_{4k}-z_{4k}-p^*u=\delta v.
$$
Then we have
$$
i^* z'_{4k}=i^*z_{4k}+\delta i^* v
$$
because $i^*p^* u=0$. Hence
$$
[i^* z'_{4k}]=[i^*z_{4k}]\in H^{4k}(\mathrm{IOut}_n;\R)
$$
as required.

Next we prove the latter part claiming that this cohomology class is $\mathrm{GL}(n,\Z)$-invariant.
For this, it is enough to prove the following. Any element
$\varphi\in\mathrm{Out}\,F_n$ induces an automorphism 
$\iota_\varphi$ of $\mathrm{IOut}_n$
by the correspondence
$$
\mathrm{IOut}_n\ni\psi\mapsto \iota_\varphi(\psi)=\varphi \psi\varphi^{-1}\in \mathrm{IOut}_n.
$$
Then, under the induced automorphism
$$
\iota_\varphi^*: H^*(\mathrm{IOut}_n;\R)\cong H^*(\mathrm{IOut}_n;\R),
$$
the equality
$$
\iota_\varphi^*([i^* z_{4k}])=[i^* z_{4k}]
$$
holds. To prove this, observe first that the cohomology class $\iota_\varphi^*([i^* z_{4k}])$
is represented by the cocycle
$
i^*\iota_\varphi^*z_{4k}
$
which is the restriction to $\mathrm{IOut}_n$ of the cochain
$$
\iota_\varphi^*z_{4k}\in Z^{4k}(\mathrm{Out}\,F_n;\R).
$$
This cochain in turn satisfies the identity
$$
\delta (\iota_\varphi^*z_{4k})=\iota_\varphi^* p^*b_{2k+1}.
$$
If we denote by $\iota_{\bar{\varphi}}$
the inner automorphism of $\mathrm{GL}(n,\Z)$ induced by the projected element
$\bar{\varphi}=p(\varphi) \in \mathrm{GL}(n,\Z)$, then we have the
following commutative diagram
$$
\begin{CD}
  \mathrm{Out}\,F_n@>{\iota_\varphi}>>  \mathrm{Out}\,F_{n}\\
@ V{p}VV @V{p}VV\\
 \mathrm{GL}(n,\Z)@>{\iota_{\bar{\varphi}}}>>  \mathrm{GL}(n,\Z).
\end{CD}
$$
Then we have
$
\iota_\varphi^* p^*b_{2k+1}=p^* \iota^*_{\bar{\varphi}} b_{2k+1}
$
and hence
$$
\delta (\iota_\varphi^*z_{4k})=p^* \iota^*_{\bar{\varphi}} b_{2k+1}.
$$
Now, as is well known, any inner automorphism of any group induces the
identity on its (co)homology group. Therefore, the cochain $\iota^*_{\bar{\varphi}} b_{2k+1}$
of the group $\mathrm{GL}(n,\Z)$ is cohomologous to $b_{2k+1}$.
Hence, by replacing $b'_{2k+1}$ and $z'_{4k}$
with $\iota^*_{\bar{\varphi}} b_{2k+1}$ and $\iota_\varphi^*z_{4k}$ respectively, in the former argument
above, we can conclude that
$$
\iota_{\varphi}^*([i^*z_{4k}])=[i^* \iota_\varphi^*z_{4k}]=[i^*z_{4k}]
$$
as required. This completes the proof.
\end{proof}

\begin{definition}
$$
T\beta_{2k+1}=[i^*z_{4k}]\in H^{4k}(\mathrm{IOut}_n;\R)^{\mathrm{GL}(n,\Z)}
$$
$$
T\beta^0_{2k+1}=q^*[i^*z_{4k}]\in H^{4k}(\mathrm{IA}_n;\R)^{\mathrm{GL}(n,\Z)}
$$
where
$$
q: \mathrm{IA}_n\rightarrow \mathrm{IOut}_n
$$
denotes the natural projection. By the above construction, we see that our
secondary class $T\beta^0_{2k+1}$ is stable in the following sense. Namely,
if we denote by
$$
i: \mathrm{IA}_n\rightarrow \mathrm{IA}_{n+1}
$$
the natural inclusion in the stable range, then we have
$$
i^* T\beta^0_{2k+1}=T\beta^0_{2k+1}.
$$
It follows that we can define this class for {\it all} $n$ by just pulling back the above
stable class by the natural inclusion $\mathrm{IA}_n\subset \mathrm{IA}_{N}$
where $N$ is a large number.
\label{def:Tb}
\end{definition}

\begin{prop}
The class $T\beta^0_{2k+1}$ is primitive in the following sense. Namely, if we denote
by 
$$
\mu_0: \mathrm{IA}_n\times  \mathrm{IA}_{n'}\rightarrow \mathrm{IA}_{n+n'}
$$
the natural homomorphism, then we have
$$
\mu_0^*(T\beta^0_{2k+1})=T\beta^0_{2k+1}\otimes 1+1\otimes T\beta^0_{2k+1}.
$$
\label{prop:Tbetap}
\end{prop}
\begin{proof}
Consider the following commutative diagram
\begin{equation}
\begin{CD}
  \mathrm{IA}_n\times   \mathrm{IA}_{n'}@>{\mu_{0}}>>  \mathrm{IA}_{n+n'}\\
  @ V{i\times i}VV @V{i}VV\\
  \mathrm{Aut}\,F_n\times   \mathrm{Aut}\,F_{n'}@>{\mu}>>  \mathrm{Aut}\,F_{n+n'}\\
@ V{p\times p}VV @V{p}VV\\
 \mathrm{GL}(n,\Z)\times   \mathrm{GL}(n',\Z)@>{\bar{\mu}}>>  \mathrm{GL}(n+n',\Z).
\end{CD}
\label{eq:cd}
\end{equation}
By Proposition \ref{prop:betap}, we have
$$
\bar{\mu}^* \beta_{2k+1}=\beta_{2k+1}\otimes 1+1\otimes \beta_{2k+1}.
$$
It follows that there exists a cochain $d \in C^{4k-1}(\mathrm{GL}(n
,\Z)\times   \mathrm{GL}(n',\Z);\R)$ such that
$$
\bar{\mu}^* b_{2k+1}=b_{2k+1}\times 1+1\times b_{2k+1}+\delta d.
$$
Then we have
\begin{align*}
(p\times p)^*\bar{\mu}^* b_{2k+1}&=\delta z_{4k}\times 1 + 1\times \delta z_{4k}+(p\times p)^*\delta d\\
&=\delta\left(z_{4k}\times 1+1\times z_{4k}+(p\times p)^* d\right).
\end{align*}
Now, by the definition of $T\beta_{2k+1}$
$$
p^* b_{2k+1}=\delta z_{4k}\ \text{and}\ T\beta_{2k+1}=[i^* z_{4k}]
$$
so that
$$
\mu^* p^* b_{2k+1}=\delta \mu^*z_{4k}.
$$
Since $p\circ \mu=\bar{\mu}\circ p\times p$, we can conclude
$$
\delta\left(z_{4k}\times 1+1\times z_{4k}+(p\times p)^* d\right)=\delta \mu^*z_{4k}
$$
and hence
$$
\delta(\mu^*z_{4k}-z_{4k}\times 1-1\times z_{4k}-(p\times p)^* d)=0.
$$
Thus, the element in the parenthesis above is a cocycle of the group
$\mathrm{Aut}\,F_n\times   \mathrm{Aut}\,F_{n'}$.
In the stable range, where $n$ and $n'$ are sufficiently large
$$
H^{4k}(\mathrm{Aut}\,F_n\times   \mathrm{Aut}\,F_{n'};\R)=0
$$
by the vanishing theorem of Galatius. Therefore, there exists an
element 
$$
d'\in C^{4k-1}(\mathrm{Aut}\,F_n\times   \mathrm{Aut}\,F_{n'};\R)
$$
such that
$$
\mu^*z_{4k}-z_{4k}\times 1-1\times z_{4k}-(p\times p)^* d=\delta d'.
$$
Hence
$$
\mu_{0}^*i^* z_{4k}=(i\times i)^* \mu^*z_{4k}=i^* z_{4k}\times 1+1\times i^*z_{4k}+\delta (i\times i)^* d'.
$$
We now conclude that
$$
\mu_{0}^*[i^* z_{4k}]=[i^* z_{4k}]\otimes 1+1\otimes [i^*z_{4k}].
$$
This competes the proof.
\end{proof}

\begin{remark}
The above proposition is a corollary of the following theorem and 
general property of the higher $\mathrm{FR}$ torsion (see Theorem 5.7.5 of Igusa \cite{i}).
However, for completeness, we gave a proof in the framework of this paper.
In the next section,  it will be extended to the case of mapping class group (Proposition \ref{prop:hbetap}).
See also Remark \ref{rem:bhat}.
\end{remark}

\begin{thm}
The secondary class $T\beta_{2k+1}$ is equal to 
Igusa's higher FR torsion class
$\tau_{2k}(\mathrm{IOut}_{n})$.
\label{thm:Tb}
\end{thm}

\begin{proof}
Proof is given by putting the vanishing theorem of Galatius 
in Igusa's theory of higher $\mathrm{FR}$ torsions developed in \cite{i}.
More precisely, let us consider the following homotopy commutative diagram
$$
\begin{CD}
{} @. {}   \Omega\mathrm{BGL}(\infty,\Z)^+\\
  @.  @V{\bar{i}_0}VV \\
  \mathrm{BIOut}_n @>{f_0}>>   |\mathcal{W}h^h_{\cdot}(\Z,1)|\\\
  @ V{\mathrm{B}i}VV @V{\bar{i}}VV \\
 \mathrm{BOut}\,F_n@>{f}>>   \Z\times\mathrm{BOut}^+_\infty\overset{h.e.}{\cong}QS^0\\
   @ V{\mathrm{B}p}VV @V{\bar{p}}VV \\
 \mathrm{BGL}(n,\Z)@>{\bar{f}}>>   \Z\times \mathrm{BGL}(\infty,\Z)^+
 \end{CD}
$$
described in Proposition 8.5.6 of the above cited book.
The point here is that we can put Galatius'
result
$$
\Z\times\mathrm{BOut}^+_\infty\overset{h.e.}{\cong}QS^0
$$
(see Theorem \ref{thm:galatius}) in the third place from the top of the right column.
Each of the two successive three spaces appearing in the right column is a fibration sequence
and each connected component of $QS^0$ is rationally trivial.
Hence the map $\bar{i}_0$ on the right column is a rational homotopy equivalence.
Now Igusa's higher torsion 
is defined roughly as follows. He constructs an explicit $4k$-cocycle of the Volodin space
$V(\Z)$, which is homotopy equivalent to $\Omega\mathrm{BGL}(\infty,\Z)^+$,
such that its cohomology class in $H^{4k}(\Omega\mathrm{BGL}(\infty,\Z)^+;\R)$
corresponds to the Borel class $\beta_{2k+1}$.
Since $\bar{i}_0$ is a rational homotopy equivalence as above, this cohomology class induces the
{\it universal} higher $\mathrm{FR}$ torsion class
$$
\tau_{2k}\in H^{4k}(|\mathcal{W}h^h_{\cdot}(\Z,1)|;\R).
$$
Then his torsion $\tau_{2k}(\mathrm{IOut}_n)\in H^{4k}(\mathrm{IOut}_n;\R)$ is defined to be
the image under $f_0^*$ of the above universal class. Our claim now follows by simply 
comparing the spectral sequence for the rational cohomology groups of the fibration
given by the lower three terms of the right column with that of the path fibration over 
$\mathrm{BGL}(\infty,\Z)^+$.
\end{proof}

\begin{problem}
Construct a cochain $z_{4k}\in C^{4k}(\mathrm{Out}\,F_n;\R)$ such that
$\delta z_{4k}=p^*b^H_{2k+1}$ explicitly. Here $b^H_{2k+1}$ denotes Hamida's cocycle
given in \cite{hamida} which
represents the Borel class $\beta_{2k+1}$.
We can also consider another cocycle for $\beta_{2k+1}$ along the line
of Dupont \cite{d}.
\end{problem}

\section{Secondary characteristic classes for the mapping class group}\label{sec:scm}

In this section, we define secondary classes for the mapping class group by
comparing two different ways of vanishing of 
the Borel regulator classes $\beta_{2k+1}$, one on the automorphism group of free groups and
the other on the Siegel modular group. 
We show that they are non-zero multiples of
the $\mathrm{MMM}$ classes $e_{2k}$ of {\it even} indices
by relating them to Igusa's higher torsions for the mapping class groups.
This would give a new geometric meaning to these classes
from the viewpoint of the theory of cohomology of groups.

We mention here that, about differences between $e_{odd}$ 
and $e_{even}$ classes, 
there are several interesting results. 
Church, Farb and Thibault \cite{cft} proved that the former classes
have some nice geometric property which the latter classes do not. On the other hand, 
Giansiracusa and Tillmann \cite{gt} and Sakasai \cite{sakasai12}
proved that the former classes vanish on the handlebody subgroup and
Lagrangian subgroup of the mapping class group, respectively.
More strongly, 
Hatcher \cite{hatcher} proved that the stable rational cohomology
of the handlebody subgroup 
is the polynomial algebra
generated by $e_{even}$  classes.
Also it was shown in \cite{morita99} that 
these classes represent the 
{\it orbifold} Pontrjagin classes
of the moduli space of curves.

Now let $\mathcal{M}_{g,1}$ denote the mapping class group of a compact oriented genus $g$ surface
with one boundary component as before
and let
$$
i: \mathcal{M}_{g,1}\rightarrow \mathrm{Aut}\, F_{2g}
$$
be the inclusion given by the classical theorem of Dehn-Nielsen-Zieschang.
It induces a related inclusion
$$
i: \mathcal{M}_{g,*}\rightarrow \mathrm{Out}\, F_{2g}
$$
and we have the following commutative diagrams.
$$
\begin{CD}
 \mathcal{I}_{g,*} @>{j_{0}}>> \mathrm{IOut}_{2g}@. \quad\quad   \mathcal{M}_{g,1} @>{\tilde{j}}>>  \mathrm{Aut}\,F_{2g} \\
 @V{i_0}VV @V{i}VV \quad\quad  @V{q}VV @V{\bar{q}}VV\\
 \mathcal{M}_{g,*} @>{j}>> \mathrm{Out}\,F_{2g}@. \quad\quad  \mathcal{M}_{g,*} @>{j}>> \mathrm{Out}\,F_{2g}\\
 @V{p_0}VV @V{i}VV \quad\quad  @V{p_0}VV @V{p}VV\\
 \mathrm{Sp}(2g,\Z) @>{\bar{j}}>> \mathrm{GL}(2g,\Z),@.  \quad\quad  \mathrm{Sp}(2g,\Z) @>{\bar{j}}>> \mathrm{GL}(2g,\Z).\\
\end{CD}
$$

As in the previous section, there exists a cochain $z_{4k}\in C^{4k}(\mathrm{Out}\, F_{2g};\R)$ such that
$$
p^*b_{2k+1}=\delta z_{4k}.
$$
On the other hand, the stable cohomology of $\mathrm{Sp}(2g,\Z)$ is a polynomial algebra
on $c_1,c_3,\ldots$ (see Theorem \ref{thm:borel}) so that there are no cohomology classes of odd degrees. Hence
there exists a cochain $y_{4k}\in C^{4k}(\mathrm{Sp}(2g,\Z);\R)$ such that
$$
\bar{j}^*(b_{2k+1})=\delta y_{4k}.
$$
Then we have
$$
0=j^*(p^*b_{2k+1})-p_0^*(\bar{j}^* b_{2k+1})=\delta (j^* z_{4k}-p_0^* y_{4k})
$$
so that 
$$
j^* z_{4k}-p_0^* y_{4k}\in Z^{4k}(\mathcal{M}_{g,*};\R).
$$
In the stable range, the cohomology class represented by this cocycle
can be written as
$$
[j^* z_{4k}-p_0^* y_{4k}]=f(e_1,\ldots,e_{2k})+e g(e,e_1,\ldots,e_{2k-1})\in H^{4k}(\mathcal{M}_{g,*};\R)
$$
by Theorem \ref{thm:looijenga}.
Also, we have
$$
q^*[j^* z_{4k}-p_0^* y_{4k}]=f(e_1,\ldots,e_{2k})\in H^{4k}(\mathcal{M}_{g,1};\R)
$$
where
$$
q: \mathcal{M}_{g,1}\rightarrow \mathcal{M}_{g,*}
$$
denotes the natural projection.

Now it was proved in \cite{morita87}\cite{mumford} that the homomorphism
$$
p_0^*: H^*(\mathrm{Sp}(2g,\Z);\Q)\rightarrow H^*(\mathcal{M}_{g,*};\Q)
$$
is injective, in a certain stable range, and its image is equal to the
subalgebra generated by the $\mathrm{MMM}$ classes of {\it odd} indices.
In fact, the pull back of the Chern classes are those of the Hodge bundle over the
moduli space and the totality of them is the same as that of
the $\mathrm{MMM}$ classes of {\it odd} indices.
Hence, by adding suitable cocycles in $Z^{4k}(\mathrm{Sp}(2g,\Z);\R)$,
we may assume that the polynomial $f(e_1,\ldots,e_{2k})$ does not 
contain monomials of $e_{odd}$ classes.

\begin{prop}
The cohomology class $[j^* z_{4k}-p_0^* y_{4k}]\in H^{4k}(\mathcal{M}_{g,*};\R)$ 
is well-defined independent of the choices of $b_{2k+1}$,  $z_{4k}$ and $y_{4k}$.
\end{prop}

\begin{proof}
Let $b'_{2k+1}\in Z^{4k+1}(\mathrm{GL}(2g,\Z);\R)$ be another representative of 
the Borel class $\beta_{2k+1}$ and let
$$
z'_{4k}\in C^{4k}(\mathrm{Out}\, F_{2g};\R),\quad y'_{4k}\in C^{4k}(\mathrm{Sp}(2g,\Z);\R)
$$
be cochains such that $\delta z'_{4k}= p^*b'_{2k+1}$ and $\delta y'_{4k}=\bar{j}^* b'_{2k+1}$.  
We further assume that, in the expression of the cohomology class
$$
[j^*z'_{4k}-p_0^*y'_{4k}]=f'(e_1,\ldots,e_{2k})+e g'(e,e_1,\ldots,e_{2k-1})\in H^{4k}(\mathcal{M}_{g,*};\R),
$$
the polynomial $f'(e_1,\ldots,e_{2k})$ does not contain any monomial of $e_{odd}$ classes.

Then as in the proof of Proposition \ref{prop:binv},
there exist elements $u\in C^{4k}(\mathrm{GL}(2g,\Z);\R)$ 
and $v\in C^{4k-1}(\mathrm{Out}\,F_{2g};\R)$
such that
$$
b'_{2k+1}=b_{2k+1}+\delta u,\quad
z'_{4k}=z_{4k}+p^*u+\delta v.
$$
Hence
$$
\delta y'_{4k}=\bar{j}^* b'_{2k+1}=\bar{j}^*(b_{2k+1}+\delta u)
=\delta (y_{4k}+\bar{j}^* u).
$$
Therefore, if we set
$$
w=y'_{4k}-y_{4k}-\bar{j}^* u,
$$
then $w$ is a cocycle of the group $\mathrm{Sp}(2g,\Z)$ so that we can consider its
cohomology class
$$
[w]\in H^{4k}(\mathrm{Sp}(2g,\Z);\R).
$$
Also
$$
y'_{4k}=y_{4k}+\bar{j}^* u+w.
$$
Now we have
\begin{align*}
j^*z'_{4k}-p_0^*y'_{4k}&=j^*(z_{4k}+p^*u+\delta v)-p_0^*(y_{4k}+\bar{j}^* u+w)\\
&=(j^*z_{4k}-p_0^*y_{4k})+\delta j^* v-p_0^* w
\end{align*}
because $j^*p^*u-p_0^*\bar{j}^* u=0$. It follows that
$$
[j^*z'_{4k}-p_0^*y'_{4k}]=[j^*z_{4k}-p_0^*y_{4k}]-p_0^* [w]\in H^{4k}(\mathcal{M}_{g,*};\R).
$$
By the definition of our secondary class $\hat{\beta}_{2k+1}$, both the cohomology classes
$[j^*z'_{4k}-p_0^*y'_{4k}]$ and $[j^*z_{4k}-p_0^*y_{4k}]$ do not contain any
monomial of $e_{odd}$ classes. On the other hand, as mentioned already above,
the cohomology class $p_0^* [w]$ is a linear combination of such monomials.
Hence we conclude that $p_0^* [w]=0$ and so
$$
[j^*z'_{4k}-p_0^*y'_{4k}]=[j^*z_{4k}-p_0^*y_{4k}]\in H^{4k}(\mathcal{M}_{g,*};\R)
$$
as required. This completes the proof.
\end{proof}

Based on the above discussion, we make the following definition.
\begin{definition}
$$
\hat{\beta}_{2k+1}=[j^* z_{4k}-p_0^* y_{4k}]\in H^{4k}(\mathcal{M}_{g,*};\R),
$$
$$
\hat{\beta}^0_{2k+1}=q^*(\hat{\beta}_{2k+1})\in H^{4k}(\mathcal{M}_{g,1};\R).
$$
By the above construction, we see that our
secondary class $\hat{\beta}^0_{2k+1}$ is stable in the following sense. Namely,
if we denote by
$$
i: \mathcal{M}_{g,1}\rightarrow \mathcal{M}_{g+1,1}
$$
the natural inclusion in the stable range, then we have
$$
i^* \hat{\beta}^0_{2k+1}=\hat{\beta}^0_{2k+1}.
$$
It follows that we can define this class for {\it all} $g$ by just pulling back the above
stable class by the natural inclusion $\mathcal{M}_{g,1}\subset \mathcal{M}_{G,1}$ where
$G$ is a large number.
\label{def:hb}
\end{definition}

\begin{problem}
Construct a cochain $y_{4k}\in C^{4k}(\mathrm{Sp}(2g,\Z);\R)$ such that
$\delta y_{4k}=j^*b^H_{2k+1}$ explicitly, where $b^H_{2k+1}$ denotes Hamida's cocycle
as before.
\end{problem}

\begin{remark}
Let us define subgroups $\mathrm{Aut}^\mathrm{sp}F_{2g}\subset \mathrm{Aut}\,F_{2g}$ and
$\mathrm{Out}^\mathrm{sp}F_{2g}\subset \mathrm{Out}\,F_{2g}$ by setting
\begin{align*}
\mathrm{Aut}^\mathrm{sp}F_{2g}&=\{\varphi\in \mathrm{Aut}\,F_{2g}; \varphi_*\in \mathrm{Sp}(2g,\Z)\},\\
\mathrm{Out}^\mathrm{sp}F_{2g}&=\{\varphi\in \mathrm{Out}\,F_{2g}; \varphi_*\in \mathrm{Sp}(2g,\Z)\}.
\end{align*}
These subgroups were already defined by Igusa in \cite{i}
and they are strictly larger than $\mathcal{M}_{g,1}$ and $\mathcal{M}_{g,*}$.
In fact, we have the following exact sequences
\begin{align*}
1\rightarrow \mathrm{IA}_{2g}\rightarrow &\mathrm{Aut}^\mathrm{sp}F_{2g}\rightarrow \mathrm{Sp}(2g,\Z)\rightarrow 1, \\
1\rightarrow \mathrm{IOut}_{2g}\rightarrow &\mathrm{Out}^\mathrm{sp}F_{2g}\rightarrow \mathrm{Sp}(2g,\Z)\rightarrow 1.
\end{align*}
Then our secondary characteristic classes are defined as cohomology classes of these groups
so that we can write
\begin{align*}
\hat{\beta}^0_{2k+1}&\in H^{4k}(\mathrm{Aut}^\mathrm{sp}F_{2g};\R),\\
\hat{\beta}_{2k+1}&\in H^{4k}(\mathrm{Out}^\mathrm{sp}F_{2g};\R).
\end{align*}
\end{remark}

There is a close relationship between our secondary classes $T\beta_{2k+1}, T\beta^0_{2k+1}$
and $\hat{\beta}_{2k+1},$ $\hat{\beta}^0_{2k+1}$. More precisely, we have the following result.
Let $\mathcal{I}_{g,*}\subset\mathcal{M}_{g,*}$ and $\mathcal{I}_{g,1}\subset\mathcal{M}_{g,1}$
be the Torelli subgroups of the mapping class groups. Then we have the following commutative
diagrams
$$
\begin{CD}
 \mathcal{I}_{g,*} @>{j_{0}}>> \mathrm{IOut}_{2g}@. \quad\quad  \mathcal{I}_{g,1} @>{j_{0}}>> \mathrm{IA}_{2g}\\
 @V{i_0}VV @V{i}VV \quad\quad  @V{i_0}VV @V{i}VV\\
 \mathcal{M}_{g,*} @>{j}>> \mathrm{Out}\,F_{2g},@.  \quad\quad \mathcal{M}_{g,1} @>{j}>> \mathrm{Aut}\,F_{2g}.\\
\end{CD}
$$

\begin{prop}
We have the following identities.
\begin{align*}
j_0^*(T\beta_{2k+1})&=i_0^*(\hat{\beta}_{2k+1}),\\
j_0^*(T\beta^0_{2k+1})&=i_0^*(\hat{\beta}^0_{2k+1}).
\end{align*}
\label{prop:IAM}
\end{prop}

\begin{proof}
It is enough to prove the first identity because the second one follows from it.
By the definition, we have
$$
\hat{\beta}_{2k+1}=[j^* z_{4k}-p_0^* y_{4k}].
$$
Hence
$$
i_0^*(\hat{\beta}_{2k+1})=[i_0^*j^*z_{4k}]
$$
because $i_0^*p_0^*=0$. On the other hand
$$
T\beta_{2k+1}=[i^*z_{4k}]
$$
so that
$$
j_0^*(T\beta_{2k+1})=[j_0^*i^* z_{4k}]=[i_0^*j^*z_{4k}]=i_0^*(\hat{\beta}_{2k+1})
$$
completing the proof.
\end{proof}

\begin{prop}
The class $\hat{\beta}^0_{2k+1}$ is primitive in the sense that the equality
$$
\mu_0^*(\hat{\beta}^0_{2k+1})=\hat{\beta}^0_{2k+1}\otimes 1+1\otimes \hat{\beta}^0_{2k+1}
$$
holds, where $\mu_0$ denotes the following natural mapping
$$
\mu_0: \mathcal{M}_{g,1}\times  \mathcal{M}_{g',1}\rightarrow \mathcal{M}_{g+g',1}.
$$
\label{prop:hbetap}
\end{prop}
\begin{proof}
Proof is given by refining that of Proposition \ref{prop:Tbetap}.
Consider the following commutative diagram
$$
\begin{CD}
{} @. {}   \mathcal{M}_{g,1}\times   \mathcal{M}_{g',1}
@>{j\times j}>>  \mathrm{Aut}\,F_{2g}\times   \mathrm{Aut}\,F_{2g'} \\
  @.  @V{\mu_0}VV @V{\mu}VV\\
  \mathcal{M}_{g,1}\times   \mathcal{M}_{g',1}@>{\mu_0}>>  \mathcal{M}_{g+g'}
@>{j}>>  \mathrm{Aut}\,F_{2g+2g'} \\
  @ V{p_0\times p_0}VV @V{p_0}VV @V{p}VV\\
 \mathrm{Sp}(2g,\Z)\times   \mathrm{Sp}(2g',\Z)@>{\bar{\mu}_0}>>   \mathrm{Sp}(2g+2g',\Z)
 @>{\bar{j}}>>   \mathrm{GL}(2g+2g',\Z)\\
   @ V{\bar{j}\times \bar{j}}VV @V{\bar{j}}VV @.\\
 \mathrm{GL}(2g,\Z)\times   \mathrm{GL}(2g',\Z)@>{\bar{\mu}}>>   \mathrm{GL}(2g+2g',\Z)
 @. 
 \end{CD}
$$
together with the commutative diagram \eqref{eq:cd} where we replace $n$ and $n'$ with
$2g$ and $2g'$ respectively.
By the definition
$$
\hat{\beta}^0_{2k+1}=[j^* z_{4k}-p_0^* y_{4k}]
$$
so that we have to compute
\begin{equation}
\mu_0^*\hat{\beta}^0_{2k+1}=\mu_0^*[j^* z_{4k}-p_0^* y_{4k}].
\label{eq:mhb}
\end{equation}
First, we consider the first term in the above expression.
We have proved in Proposition \ref{prop:Tbetap} that
there exist cochains 
\begin{align*}
d \in C^{4k-1}&(\mathrm{GL}(2g,\Z)\times   \mathrm{GL}(2g',\Z);\R),\\
d'\in C^{4k-1}&(\mathrm{Aut}\,F_{2g}\times   \mathrm{Aut}\,F_{2g'};\R)
\end{align*}
such that
$$
\mu^*z_{4k}=z_{4k}\times 1+1\times z_{4k}+(p\times p)^* d+\delta d'.
$$
It follows that
\begin{equation}
\begin{split}
\mu_0^* j^* z_{4k}&=(j\times j)^*\mu^* z_{4k}\\
&=j^*z_{4k}\times 1+1\times j^*z_{4k}
+(j\times j)^*(p\times p)^* d+\delta (j\times j)^*d'.
\end{split}
\label{eq:mj}
\end{equation}

Next we consider the second term of \eqref{eq:mhb}.
By the definition we have $\bar{j}^*\beta_{2k+1}=\delta y_{4k}$. Hence
\begin{equation}
\bar{\mu}_0^*\bar{j}^*b_{2k+1}=\bar{\mu}_0^*\delta y_{4k}=\delta \bar{\mu}_0^*y_{4k}.
\label{eq:bmb}
\end{equation}
On the other hand
\begin{equation}
\begin{split}
\bar{\mu}_0^*\bar{j}^*b_{2k+1}&=(\bar{j}\times\bar{j})^* \bar{\mu}^* b_{2k+1}\\
&=(\bar{j}\times\bar{j})^*(b_{2k+1}\times 1+1\times b_{2k+1}+\delta d)\\
&=\delta y_{4k}\times 1 +1\times \delta y_{4k}+\delta (\bar{j}\times\bar{j})^*d.
\end{split}
\label{eq:bmbj}
\end{equation}
From \eqref{eq:bmb} and \eqref{eq:bmbj}, we obtain
$$
\delta\left(\bar{\mu}_0^*y_{4k}-y_{4k}\times 1 -1\times y_{4k}-(\bar{j}\times\bar{j})^*d\right)=0.
$$
Therefore, if we set
$$
d^{\prime\prime}=\bar{\mu}_0^*y_{4k}-y_{4k}\times 1 -1\times y_{4k}-(\bar{j}\times\bar{j})^*d,
$$
then $d^{\prime\prime}$ is a cocycle of the group $\mathrm{Sp}(2g,\Z)\times   \mathrm{Sp}(2g',\Z)$
so that we can consider its cohomology class
$$
[d^{\prime\prime}]\in H^{4k}(\mathrm{Sp}(2g,\Z)\times   \mathrm{Sp}(2g',\Z);\R)
$$
and 
$$
\bar{\mu}_0^*y_{4k}=y_{4k}\times 1 +1\times y_{4k}+(\bar{j}\times\bar{j})^*d+d^{\prime\prime}.
$$
It follows that
\begin{equation}
\begin{split}
\mu_0^*p_0^* y_{4k}&=(p_0\times p_0)^*\bar{\mu}_0^*y_{4k}\\
&=(p_0\times p_0)^*(y_{4k}\times 1 +1\times y_{4k}+(\bar{j}\times\bar{j})^*d+d^{\prime\prime}).
\end{split}
\label{eq:m0p}
\end{equation}
By combining \eqref{eq:mj} and \eqref{eq:m0p}, we obtain
\begin{equation}
\begin{split}
\mu_0^* j^* z_{4k}-\mu_0^*p_0^* y_{4k}
=&j^*z_{4k}\times 1+1\times j^*z_{4k}
+(j\times j)^*(p\times p)^* d+\delta (j\times j)^*d'\\
&\hspace{7mm}-(p_0\times p_0)^*(y_{4k}\times 1 +1\times y_{4k}+(\bar{j}\times\bar{j})^*d+d^{\prime\prime})\\
=&(j^* z_{4k}-p_0^* y_{4k})\times 1 +1\times (j^* z_{4k}-p_0^* y_{4k})\\
& \hspace{7mm} +\delta (j\times j)^*d'-(p_0\times p_0)^*d^{\prime\prime}.
\end{split}
\label{eq:f2}
\end{equation}
Here we have used the equality

$$
(j\times j)^*(p\times p)^* d=(p_0\times p_0)^*(\bar{j}\times\bar{j})^*d
$$
which follows form the commutativity of the following diagram
$$
\begin{CD}
 \mathcal{M}_{g,1}\times   \mathcal{M}_{g',1}  @>{j\times j}>>  \mathrm{Aut}\,F_{2g}\times   \mathrm{Aut}\,F_{2g'} \\
@V{p_0\times p_0}VV @V{p\times p}VV \\
\mathrm{Sp}(2g,\Z)\times   \mathrm{Sp}(2g',\Z) @>{\bar{j}\times \bar{j}}>>  \mathrm{GL}(2g,\Z)\times   \mathrm{GL}(2g',\Z).
\end{CD}
$$
By combining \eqref{eq:mhb} with \eqref{eq:f2} above, we now conclude that
\begin{equation}
\begin{split}
\mu_0^*\hat{\beta}^0_{2k+1}&=\mu_0^*[j^* z_{4k}-p_0^* y_{4k}]\\
&=[j^* z_{4k}-p_0^* y_{4k}]\times 1 +1\times [j^* z_{4k}-p_0^* y_{4k}]
-(p_0\times p_0)^*[d^{\prime\prime}]\\
&=\hat{\beta}^0_{2k+1}\otimes 1 +1\otimes \hat{\beta}^0_{2k+1}-(p_0\times p_0)^*[d^{\prime\prime}].
\end{split}
\label{eq:fhb}
\end{equation}
By the definition of the class $\hat{\beta}^0_{2k+1}$ again, it contains no monomials 
of $\mathrm{MMM}$ classes of {\it odd} indices, namely those of 
the form $e_{2i_1-1}^{j_1}\cdots e_{2i_s-1}^{j_s}$. 
Also any $\mathrm{MMM}$ class $e_i$ is primitive (see \cite{miller} \cite{morita87}) so that
$$
\mu_0^* e_i=e_i\otimes 1+1\otimes e_i\in H^{4i}( \mathcal{M}_{g,1}\times   \mathcal{M}_{g',1};\Q).
$$
It follows that the class $\mu_0^* \hat{\beta}^0_{2k+1}$ does not contain 
any term of the following form
\begin{equation}
e_{2i_1-1}^{j_1}\cdots e_{2i_s-1}^{j_s}\otimes e_{2i'_1-1}^{j'_1}\cdots e_{2i'_t-1}^{j'_t}.
\label{eq:pure}
\end{equation}
Now as was recalled in $\S 2$, Borel \cite{borel1} proved that
$$
H^*(\mathrm{Sp}(2g,\Z);\Q)\cong \Q[c_1,c_3,\ldots]
$$
in a certain stable range.
It follows that 
$$
H^*(\mathrm{Sp}(2g,\Z)\times\mathrm{Sp}(2g,\Z);\Q) \cong \Q[c_1,c_3,\ldots]\otimes \Q[c_1,c_3,\ldots]
$$
again in a certain stable range. The cohomology class $[d^{\prime\prime}]$ 
appearing in \eqref{eq:fhb} belongs to this group
but with the {\it real} coefficients.

On the other hand, as was already mentioned above, 
the homomorphism
$$
p_0^*: H^*(\mathrm{Sp}(2g,\Z);\Q)\rightarrow H^*(\mathcal{M}_{g,1};\Q)
$$
is known to be injective, in a certain stable range, and its image is precisely the subalgebra of $\Q[e_1,e_2,\ldots]$
generated by $e_{odd}$ classes. 
Hence the class
$(p_0\times p_0)^*[d^{\prime\prime}]$ is a linear combination of terms of the
form described in \eqref{eq:pure}.

By the above argument, we finally conclude that $[d^{\prime\prime}]=0$ and hence
$$
\mu_0^*\hat{\beta}^0_{2k+1}=\hat{\beta}^0_{2k+1}\otimes 1 +1\otimes \hat{\beta}^0_{2k+1}
$$
as required. This completes the proof.
\end{proof}

The following is the main theorem of this section.

\begin{thm}
The secondary class $\hat{\beta}^0_{2k+1}\in H^{4k}(\mathcal{M}_{g,1};\R)$ is a non-zero multiple of 
the $2k$-th $\mathrm{MMM}$ class $e_{2k}$. More precisely
$$
\hat{\beta}^0_{2k+1}=(-1)^k \zeta(2k+1)\frac{e_{2k}}{2(2k)!}.
$$
\label{thm:hb}
\end{thm}

\begin{proof}
By the previous proposition, $\hat{\beta}^0_{2k+1}$ is a primitive cohomology class
contained in $H^{4k}(\mathcal{M}_{g,1};\R)$. On the other hand, by the result of 
Madsen and Weiss (Theorem \ref{thm:mw}), the stable rational cohomology group
of the mapping class group $\mathcal{M}_{g,1}$ is isomorphic to the
polynomial algebra generated by the $\mathrm{MMM}$ classes $e_i$
all of which are primitive as mentined above. Hence we conclude that
$\hat{\beta}^0_{2k+1}$ is a multiple of $e_{2k}$.

The final part of the proof, namely the non-triviality of our class
$\hat{\beta}^0_{2k+1}$, relies crucially on two results of Igusa given in \cite{i}.
One is his construction of a map
\begin{equation}
\mathrm{BOut}^{sp} F_{2g}\rightarrow |\mathcal{W}h^h_{\cdot}(\Z,1)|
\label{eq:outsp}
\end{equation}
which induces the higher torsion classes of the group $\mathrm{Out}^{sp} F_{2g}$.
The other is his determination (with Hain and Penner) of the higher torsion classes of the mapping class
group (Theorem  8.5.10 of \cite{i}).
Here Igusa mentions that, although his axiomatic higher torsion (\cite{ia}) is only defined on 
the Torelli group $\mathcal{I}_{g,*}$ and not on the mapping class group (surface bundles are not unipotent
bundles in general), a higher $\mathrm{FR}$ torsion is defined on $\mathcal{M}_{g,*}$ as well
by using, what he calls, a framed function on even fiberwise suspension of 
the total spaces of surface bundles with sections.
We refer to Igusa's books \cite{i}\cite{i2} for details.

Now we would like to consider the above mapping \eqref{eq:outsp} in our context
by examining the following homotopy commutative diagram.
$$
\begin{CD}
{} @.{} @. {} @. {}   \Omega\mathrm{BGL}(\infty,\Z)^+\\
@. @.  @.  @V{\bar{i}_0}VV \\
  \mathrm{B}\mathcal{I}_{g,*} @>{\mathrm{B}j_0}>>\mathrm{BIOut}_{2g} @=  \mathrm{BIOut}_{2g} @>{f_0}>>   |\mathcal{W}h^h_{\cdot}(\Z,1)|\\\
@ V{\mathrm{B}i_0}VV@ V{\mathrm{B}i_0^{sp}}VV  @ V{\mathrm{B}i}VV @V{\bar{i}}VV \\
 \mathrm{B}\mathcal{M}_{g,*}@>{\mathrm{B}j}>> \mathrm{BOut}^{sp} F_{2g}@>{\mathrm{B}j^{sp}}>> \mathrm{BOut}\,F_{2g}@>{f}>>  \Z\times\mathrm{BOut}^+_\infty\overset{h.e.}{\cong}QS^0\\
@ V{\mathrm{B}p_0}VV@ V{\mathrm{B}p_0^{sp}}VV   @ V{\mathrm{B}p}VV @V{\bar{p}}VV \\
 \mathrm{BSp}(2g,\Z)@= \mathrm{BSp}(2g,\Z)@>{\mathrm{B}\bar{j}}>> \mathrm{BGL}(2g,\Z)@>{\bar{f}}>>   \Z\times \mathrm{BGL}(\infty,\Z)^+ .
 \end{CD}
$$
By the result of Borel, the composed mapping $\bar{f}\circ \mathrm{B}\bar{j}$
is rationally homotopic to the constant map when we let $g$ go to the infinity. 
We choose a homotopy 
$$
H(\hspace{1mm}, t): \mathrm{BSp}(2\infty,\Z)\times I\rightarrow \Z\times \mathrm{BGL}(\infty,\Z)^+ 
$$
such that $H(\hspace{1mm}, 0)=\bar{f}\circ \mathrm{B}\bar{j}$ and $H(\hspace{1mm}, 1)$ is the constant map
and let 
$$
\widetilde{H}(\hspace{1mm}, t): \mathrm{BOut}^{sp} F_{2\infty}\times I\rightarrow \Z\times\mathrm{BOut}^+_\infty\overset{h.e.}{\cong}QS^0 
$$
be the mapping which covers $H$ and $\widetilde{H}(\hspace{1mm},0)=f\circ Bj^{sp}$.
Then the mapping $\widetilde{H}(\hspace{1mm},1)$ gives the
desired mapping \eqref{eq:outsp} (but defined only over the rationals).
In this situation, we can refine the argument of the proof of Theorem \ref{thm:Tb}.
More precisely, taking the homotopies $H, \widetilde{H}$ into account  we adjust
the cochain $j^*z_{4k}$ by subtracting $p_0^* y_{4k}$ to make a cocycle. We then
conclude that our construction of the secondary class $\hat{\beta}^0_{2k}$
``realizes"  the higher torsion classes induced by \eqref{eq:outsp} in the framework of group cocycles.

Then the non-triviality as well as the precise constant follows from the determination of the
higher torsion classes mentioned above.
\end{proof}

\begin{remark}
We have given the proof of Theorem \ref{thm:hb} above
within the framework of the theory of cohomology of groups
as much as possible, more precisely except for the non-triviality
(which is the most important property of course).
It should be desirable to have a proof of the non-triviality purely in the context of 
the present paper.
\label{rem:bhat}
\end{remark}

\begin{remark}
We call our classes {\it secondary} characteristic classes associated with the
vanishing of the Borel classes. However, the Borel classes are already
secondary classes associated with the vanishing of the Chern classes
on certain flat bundles. Therefore our classes are,
so to speak, {\it secondary secondary} classes. Our result shows that these
classes go back to the primary classes of the mapping class group,
namely the $\mathrm{MMM}$ classes (of even indices).
\label{rem:sec}
\end{remark}

\section{Conjectural geometric meaning of the Morita classes}\label{sec:cgm}

In this section, we propose a conjectural meaning of the Morita classes
$$
\mu_k\in H_{4k}(\mathrm{Out}\,F_{2k+2};\Q)\quad (k=1,2,\ldots)
$$
which were introduced in \cite{morita99} by making essential use of the
foundational works of Culler and Vogtmann \cite{cuv} and 
Kontsevich \cite{kontsevich1}\cite{kontsevich2}.
It was conjectured there that all the classes are non-trivial.
At present, only the first three classes have been proved to be non-trivial
(\cite{morita99}\cite{cov}\cite{gr}).

We expect that the Morita classes will be detected by certain
secondary classes associated with the difference between
two reasons for the vanishing of Borel regulator classes.
The definition of our secondary classes is given
in a similar way as the case of the mapping class group treated in the previous section
(although our actual development was done in the reverse order).
More precisely, in that case we made use of the vanishing of $\beta_{2k+1}$ 
on $\mathrm{Out}\,F_{n}$ as well as on $\mathrm{Sp}(2g,\Z)$, while
in the present case we replace $\mathrm{Sp}(2g,\Z)$ with
$\mathrm{GL}(2k+2,\Z)$ where $2k+2$ is our conjectural optimal
rank where $\beta_{2k+1}$ vanishes.

Let us first recall a paper \cite{lee} by Lee where he mentioned that
$\beta_{2k+1}$ does not vanish in $H^{4k+1}(\mathrm{GL}(n,\Z);\R)$
for all $n\geq 2k+3$. On the other hand, we have the following vanishing result.

\begin{thm}[Bismut-Lott \cite{bl}, Lee \cite{lee}, Franke \cite{franke}]
For any integer $k=1,2,\ldots$, the Borel regulator class
$\beta_{2k+1}$ vanishes in $H^{4k+1}(\mathrm{GL}(2k+1,\Z);\R)$.
\label{th:bl}
\end{thm}

Thus the remaining open problem for the (non-)triviality of $\beta_{2k+1}$ is as follows.

\begin{problem}
For each integer $k\geq 1$,
determine whether the Borel class $\beta_{2k+1}$ vanish in $H^{4k+1}(\mathrm{GL}(2k+2,\Z);\R)$
or not.
\end{problem}

The first two cases of the above problem have been solved.
Lee and Szczarba \cite{ls} proved that
$H^5(\mathrm{GL}(4,\Z);\Q)=0$ so that
$\beta_3=0\in H^5(\mathrm{GL}(4,\Z);\R)$. Also
Elbaz-Vincent-Gangl-Soul\'e \cite{egs} proved that
$H^9(\mathrm{GL}(6,\Z);\Q)=0$ so that
$\beta_5=0\in H^9(\mathrm{GL}(6,\Z);\R)$.

Based on these results, we would like to propose the following.
\begin{conj}
For any integer $k=1,2,\ldots$, the Borel regulator class
$\beta_{2k+1}$ vanishes in $H^{4k+1}(\mathrm{GL}(2k+2,\Z);\R)$.
\label{conj:cv}
\end{conj}

Assuming this conjecture, we define secondary cohomology classes 
$$
\overset{\circ}{\beta}_{2k+1}\in H^{4k}(\mathrm{Aut}\, F_{2k+2};\R)
$$
as follows. 
Let $b_{2k+1}\in Z^{4k+1}(\mathrm{GL}(n,\Z);\R)$ be a cocycle representing
$\beta_{2k+1}$ as before and consider the following commutative diagram.

$$
\begin{CD}
  \mathrm{Aut}\,F_{2k+2} @>{j}>> \mathrm{Aut}\,F_{n}\\
@ V{p_0}VV @V{p}VV\\
 \mathrm{GL}(2k+2,\Z) @>{\bar{j}}>> \mathrm{GL}(n,\Z).
\end{CD}
$$
Then, again as before, there exists a cochain $z_{4k}\in C^{4k}(\mathrm{Aut}\, F_n;\R)$
such that 
$$
p^*b_{2k+1}=\delta z_{4k}.
$$
Next $\bar{j}^*\beta_{2k+1}=0\in H^{4k+1}(\mathrm{GL}(2k+2,\Z);\R)$ by the 
assumption. Hence there exists a cochain $x_{4k}\in C^{4k}(\mathrm{GL}(2k+2,\Z);\R)$
such that
$$
\bar{j}^*b_{2k+1}=\delta x_{4k}.
$$
Then we have
$$
\delta(j^* z_{4k}-p_0^* x_{4k})=j^*p^*b_{2k+1}-p_0^*\bar{j}^*b_{2k+1}=0
$$
so that $(j^* z_{4k}-p_0^* x_{4k})$ is a $4k$-cocycle of the group $\mathrm{Aut}\, F_{2k+2}$.

\begin{definition}
We define
$$
\overset{\circ}{\beta}_{2k+1}=[j^* z_{4k}-p_0^* x_{4k}]\in H^{4k}(\mathrm{Aut}\, F_{2k+2};\R).
$$
\end{definition}

\begin{prop}
The cohomology class $\overset{\circ}{\beta}_{2k+1}$ 
is well-defined independent of the choices of $b_{2k+1}$,$z_{4k}$
and $x_{4k}$ modulo the indeterminacy 
$$
\mathrm{Im}\left(H^{4k}(\mathrm{GL}(2k+2,\Z);\R)\rightarrow H^{4k}(\mathrm{Aut}\,F_{2k+2};\R)\right).
$$
\label{prop:circb}
\end{prop}

\begin{proof}
First observe that we can add any $4k$-cocycle of $\mathrm{GL}(2k+2,\Z)$
to a given $x_{4k}$ so that the cohomology class $\overset{\circ}{\beta}_{2k+1}$ can
vary freely within the indeterminacy.
Now let $b'_{2k+1}, z'_{4k}, x'_{4k}$ be another set of choices so that
$$
p^*b'_{2k+1}=\delta z'_{4k},\quad \bar{j}^*b'_{2k+1}=\delta x'_{4k}.
$$
Then, as before, there exist
elements $u\in C^{4k}(\mathrm{GL}(n,\Z);\R)$ and  $v\in C^{4k-1}(\mathrm{Out}\,F_n;\R)$ such that
$$
b'_{2k+1}=b_{2k+1}+\delta u\quad \text{and}\quad z'_{4k}=z_{4k}+p^*u+\delta v.
$$
It follows that
$$
\delta x'_{4k}=\bar{j}^*b'_{2k+1}=\bar{j}^*(b_{2k+1}+\delta u)=\delta x_{4k}+\bar{j}^*\delta u
$$
and so $(x'_{4k}-x_{4k}-u)$ is a cocycle of $\mathrm{GL}(2k+2,\Z)$ and we have
$$
[x'_{4k}-x_{4k}-\bar{j}^*u]\in H^{4k}(\mathrm{GL}(2k+2,\Z);\R).
$$
On the other hand
\begin{align*}
(j^* z'_{4k}-p_0^* x'_{4k})-(j^* z_{4k}-p_0^* x_{4k})&=j^*(p^*u+\delta v)-p_0^*(x'_{4k}-x_{4k})\\
&=-p_0^*(x'_{4k}-x_{4k}-\bar{j}^*u)+\delta j^*v.
\end{align*}
Hence
$$
[j^* z'_{4k}-p_0^* x'_{4k}]=[j^* z_{4k}-p_0^* x_{4k}]-p_0^*[x'_{4k}-x_{4k}-\bar{j}^*u]
\in H^{4k}(\mathrm{Aut}\,F_{2k+2};\R).
$$
Therefore 
$$
[j^* z'_{4k}-p_0^* x'_{4k}]=[j^* z_{4k}-p_0^* x_{4k}]\quad \text{$\mathrm{mod}$ indeterminacy}
$$
as required.
\end{proof}

There is a close relation between the secondary classes $\overset{\circ}{\beta}_{2k+1}$ and $T\beta_{2k+1}$
as below, which shows that the class $\overset{\circ}{\beta}_{2k+1}$ can be interpreted as an
``extension" of $T\beta_{2k+1}$ to the whole group $\mathrm{Aut}\, F_{2k+2}$ 
at the ``critical rank" $n=2k+2$.

\begin{prop}
Let $i: \mathrm{IA}_{2k+2}\subset \mathrm{Aut}\,F_{2k+2}$ be the inclusion. Then for any choice
of $\overset{\circ}{\beta}_{2k+1}$ within its indeterminacy, we have the identity
$$
i^* \overset{\circ}{\beta}_{2k+1}=T\beta_{2k+1}\in H^{4k}(\mathrm{IA}_{2k+2};\R)^{\mathrm{GL}(2k+2,\Z)}.
$$
\end{prop}

\begin{proof}
This follows easily from the definitions of the two classes, because
$i^*p_0^* x_{4k} =0$ for any $x_{4k}$.
\end{proof}

To state our conjecture, recall that Conant and Vogtmann \cite{cov} proved that
the class $\mu_k$ has a natural lift in $H_{4k}(\mathrm{Aut}\, F_{2k+2};\Q)$
under the projection $\mathrm{Aut}\,F_{2k+2}\rightarrow \mathrm{Out}\,F_{2k+2}$
which we denote by the same letter.
We refer to \cite{kaw}
for the relation between 
the rational homology groups of $\mathrm{Aut}\,F_n$ and $\mathrm{Out}\,F_n$ in general.

\begin{conj}
For a suitable choice of $\overset{\circ}{\beta}_{2k+1}$ within the indeterminacy, we have
$$
\langle \overset{\circ}{\beta}_{2k+1},\mu_k\rangle\not=0.
$$
\label{conj:mu}
\end{conj}

\begin{remark}
If the projected image $(p_0)_*(\mu_k)\not=0\in H_{4k}(\mathrm{GL}(2k+2,\Z);\Q)$,
then the above conjecture holds ``trivially" because we can add to $\overset{\circ}{\beta}_{2k+1}$ the pull back 
under $p_0^*$ of
an element in $H^{4k}(\mathrm{GL}(2k+2,\Z);\Q)$ which detects $(p_0)_*(\mu_k)$. On the other hand,
if $(p_0)_*(\mu_k)=0$, then it is easy to see that the value $\langle \overset{\circ}{\beta}_{2k+1},\mu_k\rangle$
does not depend on the choice of $\overset{\circ}{\beta}_{2k+1}$ within the indeterminacy.
\end{remark}

We have also the following ``dual version" of the above argument. It applies to the
case $(p_0)_*(\mu_k)=0$ and is based on the following
two results. One is due to Conant and Vogtmann \cite{covs} who proved that 
$\mu_k\in H_{4k}(\mathrm{Aut}\,F_{2k+2};\Q)$ vanishes in $H_{4k}(\mathrm{Aut}\,F_{2k+3};\Q)$
under the natural inclusion $F_{2k+2}\subset F_{2k+3}$. The other is one particular
result, in a recent remarkable paper \cite{chkv} by
Conant, Hatcher, Kassabov and Vogtmann, that $\mu_k$
is supported on a certain free abelian subgroup $\Z^{4k}\subset \mathrm{Out}\, F_{2k+2}$.
Here they also give a new proof of the above vanishing result under one stabilization.

Our strategy is as follows. By using the above mentioned result in \cite{chkv},
we can construct an explicit $(2k+3)$-chain $u_f\in C_{2k+3}(\mathrm{Out}\, F_{2k+3};\Q)$
which bounds the abelian cycle $j(\Z^{4k})\subset \mathrm{Out}\, F_{2k+3}$.
On the other hand, by the assumption $(p_0)_*(\mu_k)=0$, there exists
a chain $u_b\in C_{2k+3}(\mathrm{GL}(2k+2,\Z);\Q)$ which bounds the abelian cycle
$(p_0)_*(\Z^{4k})\subset \mathrm{GL}(2k+2,\Z)$ (it can be checked that $(p_0)_*$
is injective on $\Z^{4k}$). Then consider the element
$$
u=p_*(u_f)-\bar{j}_*(u_b)\in Z_{4k+1}(\mathrm{GL}(2k+3,\Z);\Q)
$$
which is a cycle because
$$
\partial u=\partial u_f-\partial u_b=p_*j_*(\Z^{4k})-\bar{j}_* (p_0)_* (\Z^{4k})=0.
$$

\begin{conj}
For a suitable choice of $u$, we have
$$
\langle \beta_{2k+1},[u]\rangle\not=0.
$$
\label{conj:mud}
\end{conj}

\begin{prop}
If Conjecture \ref{conj:mud} holds for $k$, then $\mu_k\not= 0$.
\end{prop}
\begin{proof}
It suffices to prove the following. If $\mu_k=0$, then $\langle \beta_{2k+1},[u]\rangle =0$
for any choice of the cycle $u$. By the assumption that $\mu_k=0$, there exists a 
chain $u_0\in C_{4k+1}(\mathrm{Out}\,F_{2k+2};\Q)$ which bounds $\Z^{4k}\subset \mathrm{Out}\,F_{2k+2}$.
Then we can set $u_f=j_*(u_0)$ and $u_b=(p_0)_*(u_0)$ in the above construction of the cycle $u$,
which implies $u=0$. On the other hand, in this case the indeterminacy of $u$ comes from
adding certain $(4k+1)$-cycles of the two groups $\mathrm{Out}\, F_{2k+3}$ and $\mathrm{GL}(2k+2,\Z)$
on which the Borel class $\beta_{2k+1}$ takes the value $0$, the former group by the vanishing
theorem of Igusa and Galatius and the latter group by the assumption that Conjecture \ref{conj:cv} holds.
This completes the proof.
\end{proof}

\begin{remark}
We are planning to prove Conjecture \ref{conj:mud} for the case $k=1$ by using 
Hamida's cocycle for $\beta_{2k+1}$ given in \cite{hamida} and a computer computation
along the line of \cite{set}.
\end{remark}

\section{Prospects and final remarks}\label{sec:cp}

Here we discuss a relationship between our secondary classes $T\beta_{2k+1}, \hat{\beta}_{2k+1}$
and an important open question about the cohomology of the Torelli group $\mathcal{I}_g$ whether
the $\mathrm{MMM}$ classes of {\it even} indices are non-trivial in $H^{4*}(\mathcal{I}_g;\Q)$
or not. We mention that even the non-triviality of $e^2\in H^4(\mathcal{I}_{g,*};\Q)$ is not
known where $\mathcal{I}_{g,*}$ denotes the Torelli group of a genus $g$ surface with base point
and $e\in H^2(\mathcal{I}_{g,*};\Z)$ denotes the Euler class of the tangent bundle along the fibers of
surface bundles. See Sakasai \cite{sakasai} for an attempt to attack this problem
and Salter \cite{salter} for a more recent work.

\begin{conj}[Church and Farb \cite{cf}, Conjecture 6.5]
The $\mathrm{GL}(n,\Z)$-invariant part of the stable rational cohomology of $\mathrm{IA}_n$ vanishes.
\label{conj:cf}
\end{conj}

We would like to point out that an affirmative solution to the above conjecture implies that
all the $\mathrm{MMM}$ classes of {\it even} indices are trivial on the Torelli group.
Indeed, our secondary classes $T\beta^0_{2k+1}\in H^{4k}(\mathrm{IA}_n;\R)^{\mathrm{GL}(n,\Z)}$
are stable classes and $\mathrm{GL}(n,\Z)$-invariant (see Proposition \ref{prop:binv} and Definition \ref{def:Tb}). 
Therefore if we assume the above conjecture, then we have $T\beta^0_{2k+1}=0$. On the other hand,
Theorem \ref{thm:Tb} shows that $T\beta^0_{2k+1}$ is equal to Igusa's higher torsion $\tau_{2k}(\mathrm{IA}_n)$.
We can then use Igusa's result in \cite{i} that the restriction of $\tau_{2k}(\mathrm{IA}_n)$
to the Torelli group $\mathcal{I}_{g,1}$ is 
a non-zero multiple of $e_{2k}$ to conclude that this class vanishes.

If we consider the spectral sequence for the rational cohomology group of the extension
$$
1\rightarrow \mathrm{IA}_{n}\rightarrow \mathrm{Aut}\,F_{n}\rightarrow \mathrm{GL}(n,\Z)\rightarrow 1,
$$
the $E_2$ term is given by
$$
E_2^{p,q}=H^p(\mathrm{GL}(n,\Z);H^q(\mathrm{IA}_n;\Q))
$$
and the Borel classes are described as 
$$
\beta_{2k+1}\not=0 \in E_2^{4k+1,0}.
$$
Because of the vanishing theorem of Galatius, these classes do not survive in the $E_\infty$ term.
Hence among the following groups
\begin{equation}
H^{4k-i}(\mathrm{GL}(n,\Z);H^i(\mathrm{IA}_n;\Q))\quad (i=1,\ldots,4k),
\label{eq:ik}
\end{equation}
there should exist at least one non-trivial group that kills $\beta_{2k+1}$.
The secondary class 
$$
T\beta^0_{2k+1}=\tau_{2k}(\mathrm{IA}_n) \in H^{4k}(\mathrm{IA}_n;\R)^{\mathrm{GL}(n,\Z)}
\cong H^{0}(\mathrm{GL}(n,\Z);H^{4k}(\mathrm{IA}_n;\R))
$$
corresponds to the last case of $i=4k$.

Conjecture \ref{conj:cf}, applied to $\beta^0_{2k+1}$,
is equivalent to saying that it is not killed in the last step,
namely $T\beta^0_{2k+1}=0$ so that
it must be killed somewhere in the range $1\leq i\leq 4k-1$.
We would like to ask the validity of, so to speak,
the most optimistic possibility.

\begin{quest}
Are the secondary classes $T\beta^0_{2k+1}\in H^{4k}(\mathrm{IA}_n;\R)^{\mathrm{GL}(n,\Z)}$ 
non-trivial in a certain stable range? If not, which group  in \eqref{eq:ik} $(1\leq i\leq 4k-1)$
kills $\beta_{2k+1}$?
\end{quest}

\begin{remark}
There have been obtained several important results concerning the homology of 
the group $\mathrm{IA}_n$ including \cite{kaw}\cite{pettet}\cite{bbm}\cite{chp}\cite{dp}.
However, no non-trivial elements in $H^{*}(\mathrm{IA}_n;\Q)^{\mathrm{GL}(n,\Z)}$
were found. The problem of determining whether this group is non-trivial or not
remains a mystery.
\end{remark}

\begin{remark}
We think that the non-triviality of the Morita classes is a very important problem
not only in the theory of cohomology of automorphism groups of free groups
but also in low dimensional geometric topology.
This is because, we expect that the ``group version" defined in \cite{morita08},
of these classes 
(the Lie algebra version)
would detect
the difference between the smooth and topological categories in four-dimensional
topology (see \cite{morita06}\cite{mss8}), although it will need many more years for this expectation to be
clarified.
\end{remark}

\begin{remark}
In this paper, we defined our secondary classes
by making use of two different ways of vanishing of the Borel
regulator classes. It would be meaningful to recall
here the following work.
Namely, by making use of two different cocycles representing the 
first $\mathrm{MMM}$ class, the first named author defined in \cite{morita91} a secondary characteristic class 
which is an $\mathcal{M}_g$-invariant
homomorphism
$$
d:\mathcal{K}_g\rightarrow \Q
$$
where $\mathcal{K}_g\subset \mathcal{M}_g$ denotes the kernel of the Johnson homomorphism \cite{johnson}.

It was proved that this is a manifestation of the Casson invariant in the
structure of the mapping class group.
\end{remark}

\begin{remark}
More generally, by comparing two cocycles for the $\mathrm{MMM}$ classes $e_{2k-1}$ of {\it odd}
indices, one from $\mathrm{Sp}(2g,\Z)$ and the other given in \cite{morita96}\cite{km},
certain secondary classes
$$
d_{k}\in H^{4k-3}(\mathcal{K}_g,\Q)^{\mathcal{M}_g}\quad (k=1,2,\ldots)
$$
were defined in \cite{morita99}\cite{morita06}, where $d_1$ is the same as $d$ in the previous remark.

In this paper, we have constructed a new cocycle for the even class $e_{2k}$ which is essentially different from the one
given in \cite{morita96}\cite{km}. It would be interesting to study whether the difference between these
two cocycles will give rise to certain secondary invariants for a suitable subgroup of the mapping class group.
\end{remark}

\vspace{1cm}
\bibliographystyle{amsplain}

\begin{thebibliography}{30}

\bibitem{bbm}
M. ~Bestvina, K.-U. ~Bux, D. ~Margalit, 
\textit{Dimension of the Torelli group for $\mathrm{Out}(F_n)$}, 
Invent.\ Math.  170 (2007), 1--32. 

\bibitem{bl}
J.~Bismut,  J.~Lott, 
\textit{Torus bundles and the group cohomology of $GL(N,\Z)$}, 
J. \ Diff. \ Geom. 45 (1997) 196--236.

\bibitem{borel1}
A.~Borel, 
\textit{Stable real cohomology of arithmetic groups}, 
Ann. \ Sci. \ Ecole \ Norm. \ Sup. 7 (1974), 235--272.

\bibitem{charney}
R.~Charney, 
\textit{A generalization of a theorem of Vogtmann}, 
Jour. \ Pure and  Applied \ Alg. 44 (1987), 107--125.

\bibitem{cs}
J.~Cheeger, J.~Simons,
\textit{Differential characters and geometric invariants},
Geometry and Topology, (J. Alexander and J. Harer eds.),
Lecture Notes in Math. 1167, Springer, New York, 1985, 50--80.

\bibitem{set}
Z.~Choo, W.~Mannan, R.~S\'anchez-Garc\'ia, V.~Snaith, 
\textit{Computing Borel's regulator $I$}, preprint, 2009.

\bibitem{cf}
T.~Church, B.~Farb, 
\textit{Representation theory and homological stability}, 
Adv. Math. 245 (2013), 250--314.

\bibitem{cft}
T.~Church, B.~Farb, M. ~Thibault,
\textit{On the geometric nature of characteristic classes of surface bundles},
J. \ Topology 5 (2012), 575--592. 

\bibitem{chp}
F.~Cohen, A.~Heap, A.~Pettet, 
\textit{On the Johnson filtration of the automorphism group of a free group},
J. \ Algebra 329 (2011), 72--91.

\bibitem{chkv}
J.~Conant, A.~Hatcher, M.~Kassabov, K.~Vogtmann, 
\textit{Assembling homology classes in automorphism groups of free groups}, 
preprint, arXiv:1501.02351v2[math.AT].

\bibitem{cov}
J.~Conant, K.~Vogtmann, 
\textit{Morita classes in the homology of automorphism groups of
free groups}, Geom.\ Topol. 8 (2004) 121--138.

\bibitem{covs}
J.~Conant, K.~Vogtmann, 
\textit{Morita classes in the homology of $\mathrm{Aut}(F_n)$ vanish after one
stabilization}, Groups\ Geometry and Dynamics 2 (2008) 1471--1499.

\bibitem{cuv}
M.~Culler, K.~Vogtmann, 
\textit{Moduli of graphs and automorphisms of free groups}, 
Invent.\ Math. 84 (1986) 91--119.

\bibitem{dp}
M.~Day, A.~Putman, 
\textit{On the second homology group of the Torelli subgroup of $\mathrm{Aut}(F_n)$},
preprint, arXiv:1408.6242[math.GT].

\bibitem{d}
J.~Dupont,
\textit{Simplicial de Rham cohomology and characteristic classes of flat bundles}, 
Topology 15 (1976), 233--245.

\bibitem{dhz}
J.~Dupont, D.~Hain, S.~Zucker, 
\textit{Regulators and characteristic classes of flat bundles}, 
The Arithmetic and Geometry of Algebraic Cycles, CRM Proceedings and Lecture Notes 
24 (2000) 47--92.

\bibitem{egs}
P.~Elbaz-Vincent, H.~Gangl, C.~Soul\'e, 
\textit{Perfect forms, $K$-theory and the cohomology of modular groups}, 
Adv. \ Math. 245 (2013) 587--624.

\bibitem{franke}
J.~Franke, 
\textit{A topological model for some summand of the Eisenstein cohomology of congruence subgroups}, 
Bielefeld preprint 91-090 (1991).

\bibitem{ga}
S.~Galatius, 
\textit{Stable homology of automorphism groups of free groups}, 
Ann.\ Math 173 (2011), 705--768.

\bibitem{gt}
J.~Giansiracusa, U.~Tillmann, 
\textit{Vanishing of universal characteristic classes for handlebody subgroups and boundary bundles}, 
J. Homotopy Relat. Struct. 6 (2011), 103--112.

\bibitem{gr}
J.~Gray, 
\textit{On the homology of automorphism groups of free groups}, 
Ph.D. thesis, University of Tennessee, 2011.

\bibitem{hamida}
N.~Hamida, 
\textit{Description explicite du r\'egulateur de Borel}, 
C.\ R.\ Acad.\ Sci.\ Paris 330 (2000) 169--172.

\bibitem{harer}
J. L. Harer, \textit{
Stability of the homology of the mapping class groups of orientable surfaces},
Ann. Math. 121 (1985), 215--249. 

\bibitem{harerthird}
J. L. Harer, \textit{
The third homology group of the moduli space of curves},
Duke Math. J. 63 (1991), 25--55. 

\bibitem{hatcher}
A.~Hatcher, 
\textit{Stable homology by scanning: variations on a theorem of Galatius}, 
lecture at Geometric Group Theory: a conference in honor of Karen Vogtmann, in Luminy, Marseille, France. 
Slides at: http://www.math.utah.edu/vogtmannfest/slides/hatcher.pdf.

\bibitem{HV}
A.~Hatcher, K.~Vogtmann, 
\textit{Homology stability for outer automorphism groups of free groups}, 
Algebr. Geom. Topol. 4 (2004), 1253--1272.

\bibitem{HVW}
A.~Hatcher, K.~Vogtmann, N.~Wahl, 
\textit{Erratum to: Homology stability for outer automorphism groups of free groups}, 
Algebr. Geom. Topol. 6 (2006), 573--579.

\bibitem{i}
K.~Igusa, 
``Higher Franz-Reidemeister Torsion", 
Providence, RI: Amer. Math. Soc.,
2002.

\bibitem{i2}
K.~Igusa, 
``Higher complex torsion and the framing principle", 
Mem. Amer. Math. Soc. 177 (2005).

\bibitem{ia}
K.~Igusa, 
\textit{Axioms for higher torsion invariants of smooth bundles}, 
J. \ Topology 1 (2008), 159--186.

\bibitem{johnson}
D.~Johnson, 
\textit{An abelian quotient of the mapping class group $\mathcal{I}_g$}, 
Math.\ Ann. 249 (1980), 225--242.

\bibitem{vdKallen}
 W.~van der Kallen,
 \textit{Homology stability for linear groups}, 
 Invent. \ Math. 60 (1980),
269--295.

\bibitem{kaw}
N.~Kawazumi, 
\textit{Cohomological aspects of Magnus expansions}, arXiv:math/0505497v3.

\bibitem{km}
N.~Kawazumi, S.~Morita, 
\textit{The primary approximation to the cohomology of the moduli space of curves and stable characteristic classes}, 
Math.\ Res.\ Lett. 3 (1996), 629--642.

\bibitem{kontsevich1}
M.~Kontsevich, 
\textit{Formal (non)commutative symplectic geometry}, 
in: ``The Gel'fand Mathematical Seminars, 1990--1992'', 
Birkh\"auser, Boston (1993), 173--187. 

\bibitem{kontsevich2}
M.~Kontsevich, 
\textit{Feynman diagrams and low-dimensional topology}, 
in: ``First European Congress of Mathematics, Vol. II (Paris, 1992)'', 
Progr.\ Math. 120, Birkh\"auser, Basel (1994), 97--121.


\bibitem{lee}
R.~Lee, \textit{On unstable cohomology classes of $\mathrm{SL}_n(\Z)$}, 
Proc. Natl. Acad. Sci. USA 75 (1978), 43--44.

\bibitem{ls}
R.~Lee, R. H. ~Szczarba, 
\textit{On the torsion in $K_4(\Z)$ and $K_5(\Z)$}, 
Duke Math. J.  45 (1978), 101--129.

\bibitem{looijenga}
E.~Looijenga, 
\textit{Stable cohomology of the mapping class group with symplectic coefficients and 
of the universal Abel-Jacobi map}, J. Algebraic Geom. 5 (1996), 135--150.


\bibitem{mw} I.~Madsen, M.~Weiss, 
\textit{The stable moduli space of Riemann surfaces: Mumford's conjecture},
 Ann. \ Math. (2) 165 (2007), 843--94.
 
\bibitem{miller} 
E.Y. ~Miller, 
\textit{The homology of the mapping class group}, 
J.\ Diff. \ Geom. 24 (1986), 1--14.

\bibitem{morita87} 
S.~Morita, 
\textit{Characteristic classes of surface bundles}, 
Invent.\ Math. 90 (1987), 551--577.

\bibitem{morita91} 
S.~Morita, 
\textit{On the structure of the Torelli group and the Casson invariant}, 
Topology 30 (1991), 603--621.

\bibitem{morita96} S.~Morita, 
\textit{A linear representation of the mapping class group of orientable surfaces and characteristic classes of surface bundles}, 
from: ``Proceedings of the Taniguchi Symposium on Topology and Teichm\"uller Spaces", 
held in Finland, July 1995, World Scientific (1996), 159--186.


\bibitem{morita99} S.~Morita, 
\textit{Structure of the mapping class groups 
of surfaces: a survey and a prospect}, 
Geometry and Topology Monographs 2, 
\textit{Proceedings of the Kirbyfest} (1999), 349--406.

\bibitem{morita06} S.~Morita, 
\textit{Cohomological structure of the mapping class group
and beyond}, in: ``Problems on Mapping Class Groups
and related Topics", edited by Benson Farb,
Proc.\ Sympos.\ Pure\ Math. 74 (2006), 
American Mathematical Society, 
329--354.

\bibitem{morita08} 
S.~Morita, 
\textit{Symplectic automorphism groups of nilpotent quotients of fundamental
groups of surfaces}, 
in ``Groups of Diffeomorphisms",
Adv. Stud. Pure Math. 52  (2008),
443--468.

\bibitem{mss2}
S.~Morita, T.~Sakasai, M.~Suzuki, 
\textit{Computations in formal symplectic geometry and characteristic classes of
moduli spaces}, Quantum Topology 6 (2015), 139--182. 

\bibitem{mss8}
S.~Morita, T.~Sakasai, M.~Suzuki, 
\textit{Characteristic classes of families of homology cylinders and four-dimensional topology}, 
in preparation. 

\bibitem{mumford} D.~Mumford, 
\textit{Towards an enumerative geometry of the moduli space of curves}, 
in ``Arithmetic and Geometry",
Progr. Math. 36 (1983), 271--328.

\bibitem{pettet}
A.~Pettet, 
\textit{The Johnson homomorphism and the second cohomology of $\mathrm{IA}_n$}, 
Algebr. Geom. Topol. 5 (2005), 725--740. 

\bibitem{sakasai}
T.~Sakasai, 
\textit{The Johnson homomorphism and the third rational cohomology group of the Torelli group}, 
Topol. Appl. 148 (2005), 83--111.

\bibitem{sakasai12}
T.~Sakasai, 
\textit{Lagrangian mapping class groups from a group homological point of view},
Algebr. Geom. Topol. 12 (2012), 267--291.

\bibitem{salter}
N.~Salter, 
\textit{Cup products in surface bundles, higher Johnson invariants, and MMM classes},
preprint,
arXiv:1601.07836[math.GT].


\bibitem{suslin}
A. A.~Suslin,
  {\textit Stability in algebraic K-theory. Algebraic K-theory, Part I} 
  (Oberwolfach, 1980), 304--333, 
  Lecture Notes in Math., 966, Springer, Berlin-New York, 1982.

\bibitem{wahl1}
 N.~Wahl, 
  \textit{Homological stability for mapping class groups of surfaces},
    in: ``Handbook of Moduli, Vol. III'', 
 Advanced Lectures in Mathematics (2012), 547--583.




\end{thebibliography}

\end{document}